\documentclass[11pt,reqno]{amsart}
\usepackage{amscd,amsmath,amssymb,amsthm,amsfonts,epsfig,graphics}
\usepackage{geometry}
\usepackage{graphicx}
\usepackage{mathrsfs,amssymb}
\usepackage{bm}
\usepackage{url}
\usepackage{subfigure}
\usepackage{hyperref}
\usepackage{color}

\theoremstyle{plain}

\newtheorem{conjecture}{Conjecture}
\newtheorem{corollary}{Corollary}

\newtheorem{lemma}{Lemma}

\newtheorem{remark}{Remark}

\newtheorem{theorem}{Theorem}
\numberwithin{equation}{section}

\newcommand{\eps}{\epsilon}

\newcommand{\ls}{\lesssim}
\newcommand{\gs}{\gtrsim}

\allowdisplaybreaks[4]
\begin{document}
\title[Restriction of toral eigenfunctions]
 {Restriction of toral eigenfunctions to totally geodesic submanifolds}

\author{Xiaoqi Huang and Cheng Zhang}


\address{Department of Mathematics, Johns Hopkins University, Baltimore, MD 21218, United States}
\email{xhuang49@jhu.edu; czhang67@jhu.edu}


\keywords{}

\dedicatory{}

\begin{abstract}
We estimate the $L^2$ norm of the restriction to a totally geodesic submanifold of the eigenfunctions of the Laplace-Beltrami operator on the standard flat torus $\mathbb{T}^d$, $d\ge2$. We reduce getting correct bounds to counting lattice points in the intersection of some $\nu$-transverse bands on the sphere. Moreover, we prove the correct bounds for rational totally geodesic submanifolds of arbitrary codimension. In particular, we verify the conjecture of Bourgain-Rudnick on $L^2$-restriction estimates for rational hyperplanes. On $\mathbb{T}^2$, we prove the uniform $L^2$ restriction bounds for closed geodesics. On $\mathbb{T}^3$, we obtain explicit $L^2$ restriction estimates for the totally geodesic submanifolds, which improve the corresponding results by Burq-G\'erard-Tzvetkov, Hu, Chen-Sogge.
\end{abstract}

\maketitle
\section{Introduction}
Let $M$ be a smooth Riemannian manifold without boundary of dimension $d\ (d\ge2)$, $\Delta$ the corresponding Laplace-Beltrami operator and $\Sigma$ a smooth embedded submanifold of dimension $k$. Burq, G\'erard and Tzvetkov \cite{BGT2006}, Hu \cite{Hu2009} established bounds for the $L^2$ norm of the restriction of eigenfunctions of $\Delta$ to the submanifold $\Sigma$, showing that if $-\Delta e_\lambda=\lambda^2 e_\lambda$, $\lambda>1$, then there exists a constant $C>0$ such that
\begin{equation}\label{BGT01}
  \|e_\lambda\|_{L^2(\Sigma)}\le C\lambda^\frac{d-k-1}2\|e_\lambda\|_{L^2(M)},\ {\rm when}\ 1\le k\le d-3,
\end{equation}
\begin{equation}\label{BGT02}
  \|e_\lambda\|_{L^2(\Sigma)}\le C\lambda^\frac12(\log\lambda)^\frac12\|e_\lambda\|_{L^2(M)},\ {\rm when}\ k=d-2,
\end{equation}
\begin{equation}\label{BGT03}
  \|e_\lambda\|_{L^2(\Sigma)}\le C\lambda^\frac14\|e_\lambda\|_{L^2(M)},\ {\rm when}\ k=d-1.
\end{equation}
All these estimates are sharp on the standard sphere $S^d$, except for the log loss. See \cite{Tat98} for the first appearance of this type of estimate. For the generalization to quasimodes, one may refer to \cite{Tacy09} and \cite{HT12}.

It was proven by Burq, G\'erard and Tzvetkov \cite{BGT2006} ($d=2$), Hu \cite{Hu2009}  ($d\ge2$), Hassell and Tacy \cite{HT12} that if $\Sigma$ is smooth submanifold of dimension $d-1$ with positive (or negative) definite second fundamental form, then the bound $\lambda^\frac14$ in \eqref{BGT03} can be improved to be $\lambda^\frac16$.  Chen and Sogge \cite{chensogge} proved that if $d=3$ and $\Sigma$ is a geodesic segment, then the factor $(\log\lambda)^\frac12$ in \eqref{BGT02} can be removed. On nonpositively or negatively curved manifolds, the improvements for these restriction estimates can be found in \cite{chensogge}, \cite{xizhang}, \cite{zhang}, \cite{blair} and references therein.

Let $\mathbb{T}^d=\mathbb{R}^d/(2\pi\mathbb{Z})^d$ ($d\ge2$) be the standard
flat torus. In \cite{BGT2006} it is observed that for the flat torus $M=\mathbb{T}^2$, \eqref{BGT03} can be improved to be
\begin{equation}\label{BGTtorus}
  \|e_\lambda\|_{L^2(\Sigma)}\le C_\eps\lambda^\eps\|e_\lambda\|_{L^2(M)},\ \forall\eps>0
\end{equation}
due to the fact that the corresponding bound on the $L^\infty$ norm of eigenfunctions holds. They raise the question whether in \eqref{BGTtorus} the bound $\lambda^\eps$ can be replaced by a constant, that is whether there is a uniform $L^2$ restriction bound. In particular, it is known to experts that if we take $\Sigma$ to be a geodesic segment on the torus, this problem is essentially equivalent to the currently open problem of whether on the circle $|x|=\lambda$, the number of lattice points on an arc of size $\lambda^\frac12$ admits a uniform bound. Furthermore, Bourgain and Rudnick \cite{BR2011} conjecture that for flat torus $M=\mathbb{T}^d\ (d\ge2)$ and real analytic hypersurface $\Sigma$, the bound $\lambda^\frac14$ in \eqref{BGT03} can be replaced by a constant.
\begin{conjecture}[Bourgain-Rudnick \cite{BR2011}]
  Let $d\ge2$ be arbitrary and $\Sigma\subset \mathbb{T}^d$ a real analytic hypersurface. Then, for some constant $C_\Sigma$, all eigenfunctions $e_\lambda$ of $\mathbb{T}^d$ satisfy
  \begin{equation}\label{BR01}
    \|e_\lambda\|_{L^2(\Sigma)}\le C_\Sigma\|e_\lambda\|_{L^2(\mathbb{T}^d)}.
  \end{equation}
  If moreover $\Sigma$ has nowhere vanishing curature and $\lambda>\lambda_\Sigma$, for some $c_\Sigma>0$, also
  \begin{equation}\label{BR02}
     \|e_\lambda\|_{L^2(\Sigma)}\ge c_\Sigma\|e_\lambda\|_{L^2(\mathbb{T}^d)}.
  \end{equation}
\end{conjecture}
 Bourgain and Rudnick \cite{BR2011} proved this conjecture when $d=2,3$ and  $\Sigma$ is a real analytic hypersurface with non-zero curvature. In higher dimension $d\ge4$, they claimed that if $\Sigma$ is a smooth hypersurface with positive definite second fundamental form, then \eqref{BGT03} can be improved:
 \[\|e_\lambda\|_{L^2(\Sigma)}\le C_\Sigma \lambda^{\frac16-\eps_d}\|e_\lambda\|_{L^2(\mathbb{T}^d)}\]
 for some $\eps_d>0$. Recently, Hezari and Rivi\`ere \cite{Hazari} verified the Bourgain-Rudnick's conjecture for a density one subsequence of eigenfunctions on any smooth hypersurface with nonvanishing principal curvatures.

In this paper we pursue the improvements of \eqref{BGT01}, \eqref{BGT02}, \eqref{BGT03} for the flat torus $M=\mathbb{T}^d\ (d\ge2)$, considering the restriction to totally geodesic submanifolds of arbitrary codimension.  We reduce getting correct bounds to counting lattice points in the intersection of some $\nu$-transverse bands on the sphere. Moreover, we prove the correct bounds for rational totally geodesic submanifolds of arbitrary codimension. In particular, we prove that the conjecture \eqref{BR01} is true if $\Sigma$ is a rational hyperplane.

To state our theorems, we introduce some notations about the lattice points on the sphere. Let $\lambda\gg 1$ and $k\ge1$. Denote the sphere
\[\lambda S^k=\{x\in \mathbb{R}^{k+1}: |x|=\lambda\}.\]
Let $\zeta\in S^k$ be a unit vector. Consider the $\lambda^\frac12$-cap $C(\lambda\zeta,\lambda^\frac12)$ which is the intersection of the sphere $\lambda S^k$ with the ball of radius $\approx \lambda^\frac12$ around $\lambda\zeta$. Let $N_{k,\lambda}$ denote the maximal number of lattice points in the $\lambda^\frac12$-cap of the sphere $\lambda S^k$. Namely,
\[N_{k,\lambda}=\max_{\zeta\in S^k}\# \mathbb{Z}^{k+1}\cap C(\lambda\zeta,\lambda^\frac12).\] The bounds on $N_{k,\lambda}$ have been studied by number theorists. Currently, the best estimates when $k=1,2$ are
\begin{equation}\label{N1}N_{1,\lambda}\ls \log\lambda,\end{equation}
\begin{equation}\label{N2}N_{2,\lambda}\ls \lambda^{\frac12-\eta}, \ \forall \ \eta<\tfrac1{32},\end{equation}
see e.g. \cite[Lemma 2.3]{BR2011}, \cite[Lemma 2.1]{BR2015}. It is expected that the correct bounds are $N_{1,\lambda}\ls 1$ and $N_{2,\lambda}\ls \lambda^\eps$. In higher dimension ($k\gg4$), the distribution of lattice points on the sphere is  close to the uniform distribution and there is a natural estimate $N_{k,\lambda}\ls \lambda^{\frac{k-2}2}$, see e.g. \cite[Appendix A]{BR2011}. Indeed, recall that the number of lattice points in $\lambda S^k$ ($k\ge4$) is $\approx \lambda^{k-1}$. Roughly speaking, the number of lattice points in the $\lambda^\frac12$-cap is
\[\approx\frac{{\rm Area\ of\ the\ cap}}{{\rm Area\ of}\ \lambda S^k}\times \lambda^{k-1}\approx \lambda^{\frac{k-2}2}.\]

Let $d\ge2$. We define a \textbf{unit band} in the sphere $\lambda S^{d-1}$ to be the subset of $\lambda S^{d-1}$ between two parallel hyperplanes with distance $\approx1$. More precisely, for $0\ne u\in\mathbb{R}^d$ and $x_0\in\mathbb{R}^d$, the unit band is defined by
\[B(u,x_0)=\{x\in\lambda S^{d-1}:\Big|\frac{u}{|u|}\cdot (x-x_0)\Big|\ls 1\}.\]

Let $A_{1,d,\lambda}$ be the maximal number of lattice points on a unit band of $\lambda S^{d-1}$. Namely,
\[A_{1,d,\lambda}=\max_{0\ne u\in\mathbb{R}^d,\ x_0\in\mathbb{R}^d}\#\mathbb{Z}^d\cap B(u,x_0).\]We call $k$ bands $B(u_1,x_{01}),...,B(u_k,x_{0k})$  are $\nu_{k,d}$-\textbf{transverse} for some $0<\nu_{k,d}\le 1$ if
 \[\frac{|u_1\wedge\cdot\cdot\cdot\wedge u_k|}{|u_1|\cdot\cdot\cdot|u_k|}\ge \nu_{k,d}.\]
  Here the length of the wedge product
  \[|u_1\wedge\cdot\cdot\cdot\wedge u_k|=\sqrt{det[u_i\cdot u_j]_{k\times k}}\ \ .\]For $d\ge3$, $2\le k\le d-1$ and fixed $\nu_{k,d}$, let  $A_{k,d,\lambda}$ be the maximal number of lattice points in the intersection of $k$ unit bands that are $\nu_{k,d}$-transverse. Specifically, we set $\nu_{k,d}=\frac1{(d-k+1)^{k/2}}$ throughout this paper.

Given fixed $k,d$ and $\nu_{k,d}$, an interesting question is to estimate the size of $A_{k,d,\lambda}$ with respect to $\lambda\gg1$. Clearly $A_{1,d,\lambda}\gs N_{d-1,\lambda}$, since a unit band reduces to a $\lambda^\frac12$-cap when one of the hyperplanes is tangent to the sphere.  In particular, we have $A_{1,2,\lambda}\approx N_{1,\lambda}$. In higher dimensions ($d\ge5$), we may roughly  have $A_{1,d,\lambda}\approx \lambda^{d-3}\gg N_{d-1,\lambda}\approx \lambda^{\frac{d-3}2}$ by recalling the bound $\approx \lambda ^{d-2}$ for the number of lattice points in $\lambda S^{d-1}$ and assuming uniform distribution of lattice points. For $2\le k\le d-1$, the intersection of $k$ transverse unit bands that contains the most lattice points may roughly look like the region in $\lambda S^{d-1}$ with size
\[\approx \underbrace{\lambda\times\cdot\cdot\cdot\times\lambda}_{d-k-1}\times \lambda^\frac12\times\underbrace{1\times\cdot\cdot\cdot\times1}_{k-1}.\]So we may naturally expect that the correct bounds for $A_{k,d,\lambda}$ are
\begin{equation}\label{Ak1}
  A_{k,d,\lambda}\ls \lambda^{d-k-2+\eps},\ {\rm when}\ 1\le k\le d-2,
\end{equation}
\begin{equation}\label{Ak2}
  A_{k,d,\lambda}\ls 1,\ {\rm when}\ k=d-1.
\end{equation}
In particular, if we assume $u_1,...,u_k$ are fixed and rational, we will see that the maximal number of lattice points in the intersection of these $k$ unit bands satisfies the bounds in \eqref{Ak1}, \eqref{Ak2}. See the proof of Theorem \ref{prop7}.

Throughout  this paper, the totally geodesic submanifolds are assumed to be bounded and have fixed unit Hausdorff measure (e.g. length, area,...).  $A\ls B$ ($A\gs B$) means $A\le CB$ ($A\ge cB$) for some positive constants $C,\ c$ independent of $\lambda$. The constants may depend on the fixed parameters, including $d,\ k,\ \eps$. $A\approx B$ means $A\ls B$ and $A\gs B$.

\subsection{2 dimensional case}
\begin{theorem}\label{prop1} Let $\gamma\subset \mathbb{T}^2$ be a geodesic segment. Then \begin{equation}\label{T2}\|e_\lambda\|_{L^2(\gamma)}\ls \sqrt{N_{1,\lambda}}\|e_\lambda\|_{L^2({\mathbb{T}^2})},\end{equation}
where the constant is independent of $\gamma$ and $\lambda$.
Moreover, for  any fixed eigenvalue $\lambda$, there exist  a geodesic segment $\gamma$ and an eigenfunction $e_\lambda$ such that
\begin{equation}\label{T2sharp}\|e_\lambda\|_{L^2(\gamma)}\approx \sqrt{N_{1,\lambda}}\|e_\lambda\|_{L^2({\mathbb{T}^2})}.\end{equation}
This means \eqref{T2} is sharp.
\end{theorem}
Therefore $\|e_\lambda\|_{L^2(\gamma)}\ls \sqrt{\log\lambda}\|e_\lambda\|_{L^2(\mathbb{T}^2)}$ by \eqref{N1}. To prove uniform $L^2$ geodesic restriction bound  on $\mathbb{T}^2$ is equivalent to prove $N_{1,\lambda}\ls1$, which is a currently open problem. Of course, this result is already known to experts (see e.g. \cite[page 1]{BR2011}), but we also give a short proof here for the sake of completeness. Note that the closed geodesics on flat torus are exactly the straight lines with rational slopes.  If we only consider the closed geodesics, we can get correct estimates.
\begin{theorem}\label{prop2} If $\gamma$ is a geodesic segment in $\mathbb{T}^2$ with fixed rational slope, then
\begin{equation}\label{T2rational}\|e_\lambda\|_{L^2(\gamma)}\ls \|e_\lambda\|_{L^2({\mathbb{T}^2})},\end{equation}
where the constant may depend on the slope but it is independent of $\lambda$. 
\end{theorem}
\begin{corollary}If $\Sigma$ is a closed geodesic in $\mathbb{T}^2$, then
\begin{equation}\|e_\lambda\|_{L^2(\Sigma)}\le C_\Sigma \|e_\lambda\|_{L^2({\mathbb{T}^2})},\end{equation}
where the constant $C_\Sigma$ is independent of $\lambda$.
\end{corollary}
If the slope of the straight line is equal to $p/q,\ gcd(p,q)=1$, then the constant in \eqref{T2rational} is $\approx \sqrt{\max\{|p|,|q|\}}$. Here $\max\{|p|,|q|\}$ is exactly the $height$ of the rational number $p/q$. A key idea to prove the uniform bound is using the $l^2$ boundedness of discrete Hilbert transform. We remark that in the second part of Theorem \ref{prop1}, the slope of the geodesic  can be chosen to be a rational number, which may depend on $\lambda$. Of course, one cannot combine \eqref{T2sharp} and \eqref{T2rational} to conclude that $N_{1,\lambda}\ls 1$, since the constant in Theorem \ref{prop2} depends on the slope of the geodesic.

\subsection{Higher dimensional cases}
\begin{theorem}\label{prop3a}Let $d\ge3$.  Let $\Sigma\subset \mathbb{T}^d$ be a totally geodesic submanifold of dimension $k$. Then
\begin{equation}\label{Tdk}\|e_\lambda\|_{L^2(\Sigma)}\ls \sqrt{A_{k,d,\lambda}}(\log\lambda)^{k/2}\|e_\lambda\|_{L^2({\mathbb{T}^d})}.\end{equation}
where the constant is independent of $\Sigma$ and $\lambda$.
Moreover, for  any fixed eigenvalue $\lambda$, there exist  a  totally geodesic submanifold $\Sigma$ and an eigenfunction $e_\lambda$ such that
\begin{equation}\label{Tdksharp}\|e_\lambda\|_{L^2(\Sigma)}\approx \sqrt{A_{k,d,\lambda}}\|e_\lambda\|_{L^2({\mathbb{T}^d})}.\end{equation}
This means \eqref{Tdk} is essentially sharp.
\end{theorem}
It is natural to expect that the  factor  $(\log\lambda)^{k/2}$ can be removed. Note that $N_{1,\lambda}\approx A_{1,2,\lambda}$. This theorem agrees with Theorem \ref{prop1} when $d=2$, except for the log loss.  In particular, when $d=3$, by estimating $A_{1,3,\lambda}$ and $A_{2,3,\lambda}$,  we can prove the following explicit estimates.
\begin{theorem}\label{prop4} Let $\Sigma\subset \mathbb{T}^3$ be a totally geodesic submanifold of dimension $k$, $k=1,2$. Then for any $\eps>0$
\begin{equation}\label{T3k1n}\|e_\lambda\|_{L^2(\Sigma)}\ls \lambda^{\frac1{3}+\eps}\|e_\lambda\|_{L^2({\mathbb{T}^3})},\ {\rm when}\ k=1,\end{equation}
\begin{equation}\label{T3k2n}\|e_\lambda\|_{L^2(\Sigma)}\ls \lambda^{\frac1{12}+\eps}\|e_\lambda\|_{L^2({\mathbb{T}^3})},\ {\rm when}\ k=2.\end{equation}
The constants are independent of $\Sigma$ and $\lambda$.
\end{theorem}
These bounds improve the bounds $\lambda^\frac12$ ($k=1$) and $\lambda^\frac14$ ($k=2$) by Burq-G\'erard-Tzvetkov, Hu, Chen-Sogge. The main idea is generalizing a result of Jarnik \cite{Jar1926} to higher dimensions. We decompose the bands into a number of small pieces, each of which contains at most $\lambda^\eps$ lattice points. All these bounds seem to be far from the correct bounds suggested by \eqref{Ak1}, \eqref{Ak2} and Theorem \ref{prop3a}, namely for any $\eps>0$
 \begin{equation}\label{T3best1}\|e_\lambda\|_{L^2(\Sigma)}\ls \lambda^\eps\|e_\lambda\|_{L^2({\mathbb{T}^3})},\ {\rm when}\ k=1,\end{equation}
\begin{equation}\label{T3best2}\|e_\lambda\|_{L^2(\Sigma)}\ls \|e_\lambda\|_{L^2({\mathbb{T}^3})},\ {\rm when}\ k=2.\end{equation}
Fortunately, these correct bounds can be proved for rational totally geodesic submanifolds (see Theorem \ref{prop7}).
\begin{remark}{\rm
  To get \eqref{T3best1}, one may need to show $A_{1,3,\lambda}\ls \lambda^\eps$. Recall that $A_{1,3,\lambda}\gs N_{2,\lambda}$. So this estimate is much stronger than the estimate $N_{2,\lambda}\ls \lambda^\eps$ for the $\lambda^\frac12$-cap, which is still an open problem. Bourgain, Rudnick and Sarnak \cite[Theorem 2.4]{BRS2016} proved that as a consequence of ``Linnik's basic Lemma'', this holds in the mean square. Indeed, partition the sphere $\lambda S^2$ into sets $C_\alpha$ of size $\lambda^\frac12$, for instance by intersecting with cubes of that size. Denote $\mathcal{E}=\mathbb{Z}^3\cap\lambda S^2$. They proved that
  \[\sum_\alpha[\#(\mathcal{E}\cap C_\alpha)]^2\ls \lambda^{1+\eps},\ \forall\eps>0.\]
  Note that there are $\approx \lambda $ terms in the sum, and $\#\mathcal{E}\ls \lambda^{1+\eps}$. This result expresses a mean-equidistribution property of $\mathcal{E}$.}
\end{remark}
\begin{theorem}\label{prop3}Let $d\ge2$ and $1\le k\le d-1$.
  For  any fixed eigenvalue $\lambda$, there exist  a totally geodesic submanifold $\Sigma$ of dimension $k$  and an eigenfunction $e_\lambda$ such that
\begin{equation}\label{Tdsharp}\|e_\lambda\|_{L^2(\Sigma)}\approx \sqrt{N_{d-k,\lambda}}\|e_\lambda\|_{L^2({\mathbb{T}^d})}.\end{equation}
\end{theorem}
This result generalizes Theorem \ref{prop1} to higher dimensions and is related to the conjecture by Bourgain-Rudnick. Indeed, for some hyperplane $\Sigma\subset\mathbb{T}^d$ and an eigenfunction $e_\lambda$, we have \[\|e_\lambda\|_{L^2(\Sigma)}\approx \sqrt{N_{1,\lambda}}\|e_\lambda\|_{L^2({\mathbb{T}^d})},\] which means that the conjecture \eqref{BR01} in any dimension ($d\ge2$) is not true if  $N_{1,\lambda}$ is unbounded.

We may generalize Theorem \ref{prop2} and prove the correct bounds in higher dimensions.
\begin{theorem}\label{prop7} Let $d\ge2$, and $\Sigma\subset \mathbb{T}^d$. If $\Sigma$ is a totally geodesic submanifold of dimension $k$ and it is determined by linear equations with rational coefficients, then for any $\eps>0$
\begin{equation}\label{Tdkrational}
  \|e_\lambda\|_{L^2(\Sigma)}\ls \lambda^{\frac{d-k-2}2+\eps}\|e_\lambda\|_{L^2({\mathbb{T}^d})},\ {\rm when} \ 1\le k\le d-2,
\end{equation}
and
  \begin{equation}\label{Tdk1ratonal}
    \|e_\lambda\|_{L^2(\Sigma)}\ls \|e_\lambda\|_{L^2({\mathbb{T}^d})}, \ {\rm when} \ k=d-1,
  \end{equation}
  where the constant may depend on the rational coefficients but is independent of $\lambda$.
\end{theorem}

Clearly, these results improve the bounds \eqref{BGT01}, \eqref{BGT02} and \eqref{BGT03}. These bounds agree with the correct bounds given by  \eqref{Ak1}, \eqref{Ak2} and Theorem \ref{prop3a}. In particular, when $k=d-1$, the uniform bound \eqref{Tdk1ratonal} agrees with the conjecture bound \eqref{BR01}. Furthermore, it is natural to make the following conjecture according to Theorem \ref{prop7}.
\begin{conjecture}
  Let $d\ge3$ be arbitrary and $\Sigma\subset \mathbb{T}^d$ a smooth submanifold of dimension $k$ ($1\le k\le d-2$). Then all eigenfunctions $e_\lambda$ of $\mathbb{T}^d$ satisfy
  \begin{equation}
    \|e_\lambda\|_{L^2(\Sigma)}\ls\lambda^{\frac{d-k-2}2+\eps}\|e_\lambda\|_{L^2(\mathbb{T}^d)},\ \forall\eps>0.
  \end{equation}
\end{conjecture}

\noindent\textbf{Acknowledgement. } The authors would like to thank Professor C. Sogge and Professor Z. Rudnick for their helpful suggestions and comments. The authors are grateful to the anonymous referee for very thorough and helpful reports.

\section{Proof of 2 dimensional case}

We prove Theorem \ref{prop1} and \ref{prop2} on $\mathbb{T}^2$.  Let $e_\lambda(x)=\sum_{n\in \mathcal{E}}c_ne^{in\cdot x}$, $\sum_{n\in\mathcal{E}}|c_n|^2=1$, where $\mathcal{E}=\mathbb{Z}^2\cap \lambda S^1$.
\subsection{Proof of Theorem \ref{prop1}}
~\\
\noindent \textbf{Proof of \eqref{T2}.}  Without loss of generality, let $\gamma=(x_1,ax_1),\ a\in\mathbb{R},\ |x_1|\le1$.  We may additionally assume that $|a|\le1$, otherwise we may exchange the $x_1,x_2$ coordinates. Then
\begin{align*}\int_\gamma |e_\lambda|^2d\sigma &=\sum_m\sum_nc_m\bar{c}_n\int_\gamma e^{i(m-n)\cdot x}d\sigma\\
&=\sum_m\sum_nc_m\bar{c}_n\int_{-1}^1e^{i(m_1-n_1+a(m_2-n_2))x_1}\sqrt{1+a^2}dx_1\\
&\ls\sum_m\sum_n|c_m||c_n|\frac{|\sin{(m_1-n_1+a(m_2-n_2))|}}{|m_1-n_1+a(m_2-n_2)|}\\
&\ls\sum_{m,n:|m_1-n_1+a(m_2-n_2)|\ls1}|c_m||c_n|+\sum_{k\ge0}^{\log_2\lambda}\sum_{m,n:|m_1-n_1+a(m_2-n_2)|\approx2^k}|c_m||c_n|2^{-k}\\
&:=I_1+I_2.\end{align*}

If $|m_1-n_1+a(m_2-n_2)|\ls1$, then for fixed $m\in\mathcal{E}$, all the possible $n\in\mathcal{E}$ lie on the two arcs between two parallel lines with distance $\approx 1$. See Figure \ref{fig11}.  Note that the length of each arc is $\ls\lambda^\frac12$.  Thus, we have at most $N_{1,\lambda}$ choices of $n$. Similarly, if $|m_1-n_1+a(m_2-n_2)|\approx 2^k$, then for fixed $m\in\mathcal{E}$, all the possible $n\in\mathcal{E}$ lie on the two arcs between two parallel lines with distance $\approx 2^k$. Note that the length of each arc is $\ls\lambda^\frac12 2^{k/2}$. Thus, we have at most $2^{k/2} N_{1,\lambda}$ choices of $n$. Therefore, by Cauchy-Schwarz
\[I_1\ls \sum_m |c_m|^2\cdot N_{1,\lambda}\approx N_{1,\lambda},\]
\[I_2\ls \sum_{k\ge0}^{\log_2\lambda}\sum_m|c_m|^22^{-k}\cdot2^{k/2} N_{1,\lambda}\approx N_{1,\lambda}.\]
This proves \eqref{T2}.\\

\begin{figure}
   \includegraphics[width=0.6\textwidth]{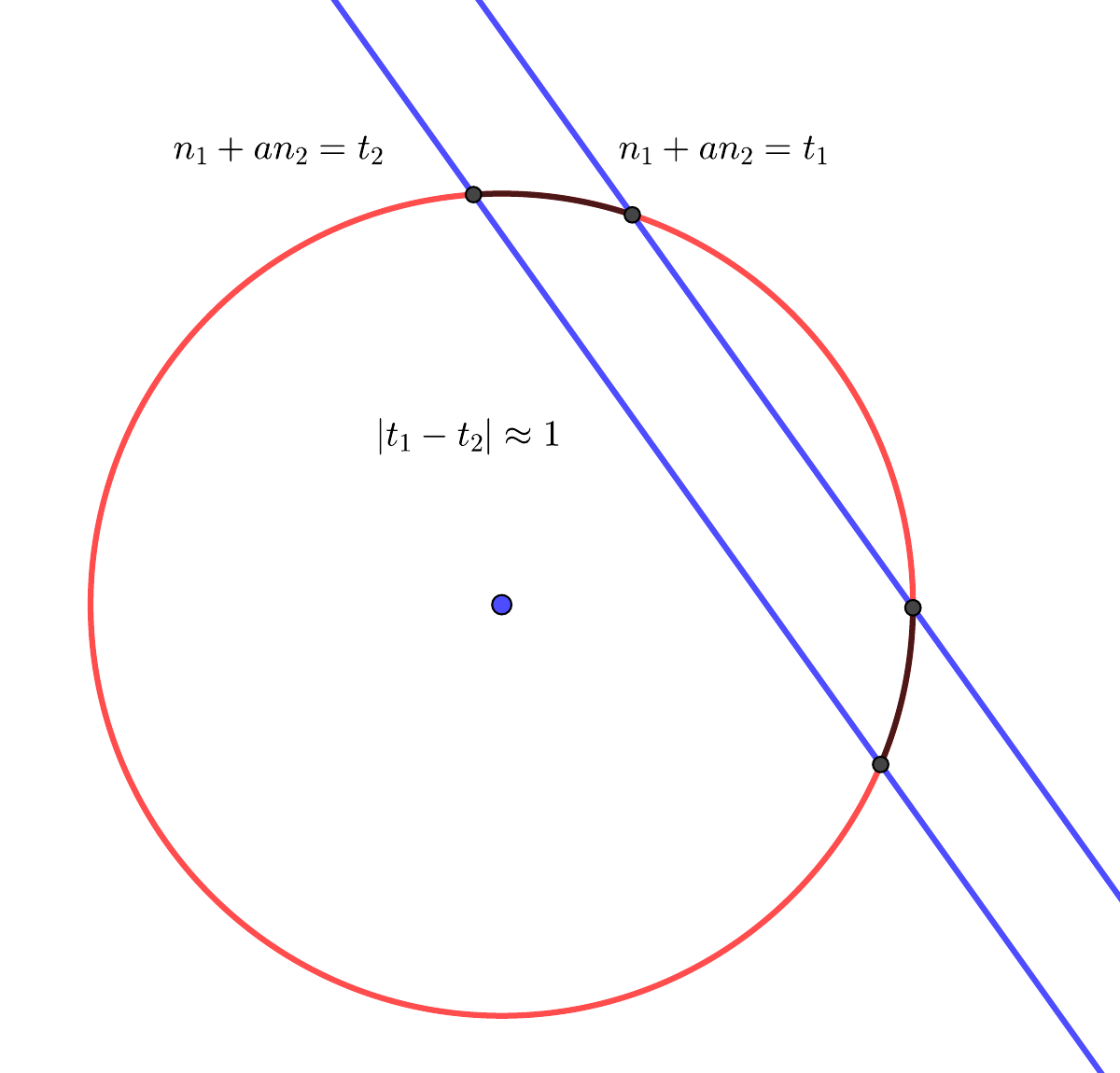}
   \caption{The arcs between the two parallel lines with distance $\approx 1$}
   \label{fig11}
   \end{figure}
  \begin{figure}
   \includegraphics[width=0.6\textwidth]{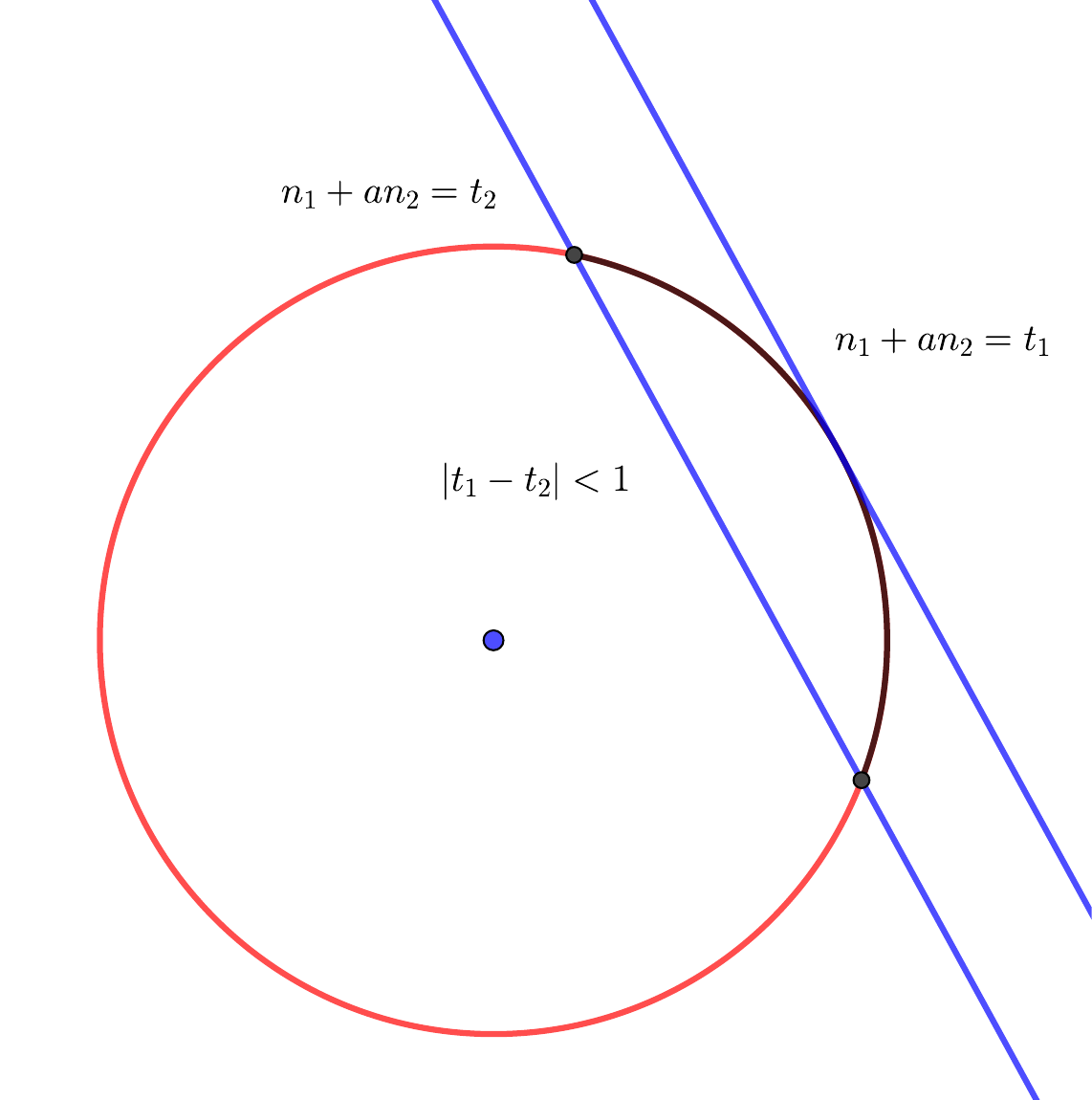}
   \caption{The $\lambda^\frac12$-arc with $\approx N_{1,\lambda}$ lattice points}
   \label{fig12}
   \end{figure}
\noindent\textbf{Proof of \eqref{T2sharp}.} We assume that there is some $\lambda^\frac12$-arc  $\mathcal{C}\subset\lambda S^1$ containing $\approx N_{1,\lambda}$ lattice points. See Figure \ref{fig12}. The arc must lie between two parallel lines: its chord and the tangent line at its midpoint. We may assume the distance between these two lines is $<\frac1{2}$. Without loss of generality, we may also assume the equations of these lines have the form: $x_1+ax_2=t$, $|a|\le1$. (Otherwise, exchange $x_1,x_2$ coordinates.) Let $c_n=\frac1{\sqrt{N_{1,\lambda}}}$ when $n$ is on the arc $\mathcal{C}$, and $c_n=0$, otherwise. Define the eigenfunction $e_\lambda(x)=\sum_{n\in\mathcal{E}}c_ne^{in\cdot x}$. Then $\sum_{n\in\mathcal{E}}|c_n|^2=1$, by definition. Since when $m,n$ are on the arc, $|m_1-n_1+a(m_2-n_2)|<\frac1{2}\sqrt{1+a^2}\le1$, so
\[0<\frac{\sin(m_1-n_1+a(m_2-n_2))}{m_1-n_1+a(m_2-n_2)}\approx 1.\] Let $\gamma=(x_1,ax_1),|x_1|\le1,|a|\le1$. Then
\begin{align*}\int_\gamma |e_\lambda|^2d\sigma &=\sum_{m\in\mathcal{E}}\sum_{n\in\mathcal{E}}c_m\bar{c}_n\int_\gamma e^{i(m-n)\cdot x}d\sigma\\
&\approx\frac1{N_{1,\lambda}}\sum_{m\in\mathcal{C}}\sum_{n\in\mathcal{C}}\frac{\sin{(m_1-n_1+a(m_2-n_2))}}{m_1-n_1+a(m_2-n_2)}\\
&\approx N_{1,\lambda}\end{align*}
Thus  \eqref{T2sharp} is proved.
\subsection{Proof of Theorem \ref{prop2}}

\begin{lemma}\label{hilbert}
  Let $|\mu|\le1$. If  $T_\mu: l^2(\mathbb{Z})\to l^2(\mathbb{Z})$
  \[(a_t)\mapsto (b_s),\ b_s=\sum_{t\in\mathbb{Z}}\frac{\sin(\mu(t-s))}{t-s}a_t,\]
  then $\sup_{|\mu|\le1}\|T_\mu\|_{l^2(\mathbb{Z})\to l^2(\mathbb{Z})}\ls 1$.
\end{lemma}
\begin{proof}
  It is a consequence of the $l^2$ boundedness of the discrete Hilbert transform, see e.g. \cite{R1927}.
\end{proof}

\noindent \textbf{Proof of \eqref{T2rational}.} Now we prove the uniform  bound for the $L^2$ restriction to $\gamma=(x_1,ax_1)$, $a=p/q,\ gcd(p,q)=1$, $|x_1|\le1$, $|a|\le1$.
Let $e_\lambda(x)=\sum_{n\in \mathcal{E}}c_ne^{in\cdot x}$, where $\mathcal{E}=\mathbb{Z}^2\cap \lambda S^1$, $\sum_{n\in\mathcal{E}}|c_n|^2=1$. Then
\begin{align*}\int_\gamma |e_\lambda|^2d\sigma &=\sum_m\sum_nc_m\bar{c}_n\int_\gamma e^{i(m-n)\cdot x}d\sigma\\
&\ls\Big|\sum_m\sum_nc_m\bar{c}_n\frac{\sin{((m_1+am_2)-(n_1+an_2))}}{(m_1+am_2)-(n_1+an_2)}\Big|.\\
\end{align*}

The straight line $x_1+ax_2=t$ intersects the circle $x_1^2+x_2^2=\lambda^2$ at at most two points, and they are separated by the perpendicular line $ax_1-x_2=0$ passing through their midpoint. Let
 \[L^+=\{(x_1,x_2)\in\mathcal{E}:ax_1-x_2\ge 0\},\]
 \[L^-=\{(x_1,x_2)\in\mathcal{E}:ax_1-x_2<0\},\]
 and then $\mathcal{E}=L^+\cup L^-$.
So we can split the sum above into 4 parts, and it suffices to estimate
\begin{align*}I_1&=\Big|\sum_{m\in L^+}\sum_{n\in L^+}c_m\bar{c}_n\frac{\sin{((m_1+am_2)-(n_1+an_2))}}{(m_1+am_2)-(n_1+an_2)}\Big|\\
&\ls \Big(\sum_{m\in L^+}\Big(\sum_{n\in L^+}\bar{c}_n\frac{\sin((m_1+am_2)-(n_1+an_2))}{(m_1+am_2)-(n_1+an_2)}\Big)^2\Big)^\frac12.\end{align*}
If $a=p/q\in\mathbb{Q}$, $gcd(p,q)=1$, $|p|\le |q|$, then for any fixed $t\in \mathbb{Z}$, there is at most one $n\in L^+$ such that $qn_1+pn_2=t$. So we may define $a_{t}:=\bar{c}_n$ if there is one $n\in L^+$ such that $qn_1+pn_2=t$, and  define $a_{t}:=0$, otherwise.
 Recall the operator $T_\mu$ in Lemma \ref{hilbert},
\[I_1\ls |q|\|T_{1/q}\|_{l^2\to l^2} (\sum_{t\in\mathbb{Z}}|a_{t}|^2)^\frac12\ls |q|.\]
Thus
\[\int_\gamma |e_\lambda|^2d\sigma\ls |q|.\]So the proof is finished.

\section{Proof of higher dimensional cases}
We prove Theorems \ref{prop3a}, \ref{prop3}, \ref{prop7} on $\mathbb{T}^d\ (d\ge2)$, generalizing the  two dimensional results to higher dimensions.  Let $e_\lambda(x)=\sum_{n\in \mathcal{E}}c_ne^{in\cdot x}$, $\sum_{n\in\mathcal{E}}|c_n|^2=1$, where $\mathcal{E}=\mathbb{Z}^d\cap \lambda S^{d-1}$, $d\ge2$.
\subsection{Proof of Theorem \ref{prop3a}}
~\\
\noindent\textbf{Proof of \eqref{Tdk}.} The following proof generalizes the method in the proof of \eqref{T2}. Let $\Sigma\subset \mathbb{T}^d\ (d\ge2)$ be a totally geodesic submanifold of dimension $k$. Note that $\Sigma$ is determined by $k$ linear equations, we can assume without loss of generality that \[\Sigma=\Big(x_1,...,x_k,\sum_{j=1}^ka_{k+1,j}x_j,...,\sum_{j=1}^ka_{d,j}x_j\Big)\]
where $|a_{ij}|\le 1$, $k+1\le i\le d,\ 1\le j\le k$.
Denote \[u_1=(1,0,...,0,a_{k+1,1},...,a_{d,1}),\]
\[u_2=(0,1,...,0,a_{k+1,2},...,a_{d,2}),\]
...\[u_k=(0,0,...,1,a_{k+1,k},...,a_{d,k}).\]
Then we can write $\Sigma=x_1u_1+...+x_ku_k$, and the measure
\[d\sigma=|u_1\wedge\cdot\cdot\cdot\wedge u_k|dx_1\cdot\cdot\cdot dx_k=\sqrt{det[u_i\cdot u_j]_{k\times k}}\ dx_1\cdot\cdot\cdot dx_k\]
and the condition $|a_{ij}|\le1$ implies that for some constant $C_{k,d}\ge1$ independent of $\Sigma$
\[1\le |u_1\wedge\cdot\cdot\cdot\wedge u_k|\le C_{k,d},\]
\[|u_j|\le (d-k+1)^{1/2},\ j=1,2,...,k.\]
Then
\begin{align*}
  \int_\Sigma|e_\lambda|^2d\sigma&=\sum_m\sum_nc_m\bar{c}_n\int_\Sigma e^{(m-n)\cdot x}d\sigma\\
  &\ls\sum_m\sum_n|c_m||c_n|\Big|\int_{[-1,1]^k}\prod_{j=1}^ke^{ix_j(m-n)\cdot u_j}dx_1\cdot\cdot\cdot dx_k\Big|\\
  &=\sum_m\sum_n|c_m||c_n|\Big|\prod_{j=1}^k\frac{\sin((m-n)\cdot u_j)}{(m-n)\cdot u_j}\Big|.\\
\end{align*}
Define for  $j=1,...,k$ and $p_j=0,1,2,...$
\begin{equation}\label{wideband}B_{p_j}(u_j,m)=\begin{cases}\{n\in\mathcal{E}:|(m-n)\cdot u_j|\ls1\}\quad {\rm if}\  p_j=0\\
\{n\in\mathcal{E}:|(m-n)\cdot u_j|\approx 2^{p_j}\}\quad {\rm if}\ p_j>0.\end{cases}\end{equation}
Therefore by Cauchy-Schwarz
\begin{align}\label{Tdkpf0}
  \int_\Sigma|e_\lambda|^2d\sigma&\ls\sum_{p_1,...,p_k=0}^{\log\lambda}\sum_{m\in\mathcal{E}}\sum_{n\in\cap_{j=1}^kB_{p_j}(u_j,m)}|c_m|^22^{-p_1}\cdot\cdot\cdot 2^{-p_k}\\
  &\label{Tdkpf}\ls \sum_{p_1,...,p_k=0}^{\log\lambda}\sum_{m\in\mathcal{E}}|c_m|^2\cdot A_{k,d,\lambda}\\
  &\ls A_{k,d,\lambda}(\log\lambda)^k.\end{align}
Here in the second inequality we split $B_{p_j}(u_j,m)$ into $\approx2^{p_j}$ unit bands $B(u_j,\cdot)$ and use the fact that the unit bands $B(u_1,\cdot),...,B(u_k,\cdot)$ are $\nu_{k,d}$-transverse
\[\frac{|u_1\wedge\cdot\cdot\cdot\wedge u_k|}{|u_1|\cdot\cdot\cdot|u_k|}\ge \frac1{(d-k+1)^{k/2}}:=\nu_{k,d}>0.\] Recall that $A_{k,d,\lambda}$ is the maximal number of lattice points in the intersection of $k$ unit bands that are $\nu_{k,d}$-transverse. Then the number of lattice points
\[\max_{m\in\mathcal{E}}\#\bigcap_{j=1}^kB_{p_j}(u_j,m)\ls A_{k,d,\lambda}2^{p_1+\cdot\cdot\cdot p_k}.\]
 This estimate gives \eqref{Tdkpf}. So the proof is complete.\\

\noindent\textbf{Proof of  \eqref{Tdksharp}.} The construction here generalizes the two dimensional case \eqref{T2sharp}. Fix an eigenvalue $\lambda>1$. Let $\mathcal{C}$ be the intersection of $k$ unit bands that are $\nu_{k,d}$-transverse,  which contains $A_{k,d,\lambda}$ lattice points in it. Without loss of generality, we may assume these bands are $B(u_1,x_{01})$,..., $B(u_k,x_{0k})$ where
 \[u_1=(1,0,...,0,a_{k+1,1},...,a_{d,1}),\]
\[u_2=(0,1,...,0,a_{k+1,2},...,a_{d,2}),\]
...\[u_k=(0,0,...,1,a_{k+1,k},...,a_{d,k}),\]
and $|a_{ij}|\le 1$, $k+1\le i\le d,\ 1\le j\le k$.

Let \[c_n=\begin{cases}\frac1{\sqrt{A_{k,d,\lambda}}},\ n\in \mathcal{C}\\
0,\ \ \ \ \ \ \ \ \ n\in\mathcal{E}\setminus\mathcal{C}.\end{cases}\]
Define the eigenfunction $e_{\lambda}(x)=\sum_{n\in\mathcal{E}}c_ne^{in\cdot x}$.
Then $\sum_{n\in\mathcal{E}}|c_n|^2=1$, by definition. If the distance between the two hyperplanes of each unit band is small enough (only depending on $k$ and $d$), then for any $m,n\in \mathcal{C}$ we have $|(m-n)\cdot u_j|\le1$, $j=1,...,k$, which implies
\[0<\frac{\sin((m-n)\cdot u_j)}{(m-n)\cdot u_j}\approx 1.\]
Let $\Sigma=x_1u_1+\cdot\cdot\cdot+x_ku_k$, $|x_j|\le1$, $j=1,...,k$. Then
\begin{align*}
  \int_{\Sigma}|e_\lambda|^2d\sigma&=\sum_{m\in\mathcal{E}}\sum_{n\in\mathcal{E}}c_m\bar{c}_n\int_\Sigma e^{i(m-n)\cdot x}d\sigma\\
  &\approx \frac1{A_{k,d,\lambda}}\sum_{m\in\mathcal{C}}\sum_{n\in\mathcal{C}}\prod_{j=1}^k\frac{\sin((m-n)\cdot u_j)}{(m-n)\cdot u_j}\\
  &\approx \frac1{A_{k,d,\lambda}}\sum_{m\in\mathcal{C}}\sum_{n\in\mathcal{C}} 1\\
  &\approx A_{k,d,\lambda}.
\end{align*}
Hence the eigenfunction $e_\lambda$ and the submanifold $\Sigma$ satisfy \eqref{Tdksharp}.

\subsection{Proof of Theorem \ref{prop3}}

The  construction here is similar to the proof of \eqref{Tdksharp}. Given fixed $d\ge2$, $1\le k\le d-1$, and an eigenvalue $\lambda>1$, we need to construct a totally geodesic submanifold $\Sigma$ of dimension $k$ and an eigenfunction $e_\lambda(x)$ satisfying \eqref{Tdsharp}. Consider the intersection of the $(d-k+1)$-plane $\Pi:\ x_{d-k+2}=...=x_d=0$ and the sphere $\lambda S^{d-1}$.  It is exactly the sphere $\lambda S^{d-k}$. Choose a $\lambda^\frac12$-cap $\mathcal{C}\subset\lambda S^{d-k}$, which contains the maximal number of lattice points $N_{d-k,\lambda}$. Then this cap must lie between two parallel $(d-k)$-planes $\pi_1,\pi_2$ in $\Pi$. Without loss of generality, we may assume the equations of these $(d-k)$-planes have the form: \[x_1+a_2x_2+...+a_{d-k+1}x_{d-k+1}=t,\ x_{d-k+2}=...=x_d=0,\] and $|a_j|\le1$, $j=2,...,d-k+1$.  We may also assume their distance is $<\frac1{d-k+1}$. Then we can construct \[\Sigma=(x_1,a_2x_1,...,a_{d-k+1}x_1,x_2,...,x_{k}),\ |x_j|\le 1,\ j=1,...,k.\]Then for any lattice points $m,n$ in this cap $\mathcal{C}$, we have
\[|m_1-n_1+...+a_{d-k+1}(m_{d-k+1}-n_{d-k+1})|<\frac1{d-k+1}\sqrt{1+a_2^2+...+a_{d-k+1}^2}\le1.\]
So we have \[0<\frac{\sin(m_1-n_1+a_2(m_2-n_2)+...+a_{d-k+1}(m_{d-k+1}-n_{d-k+1}))}{m_1-n_1+a_2(m_2-n_2)+...+a_{d-k+1}(m_{d-k+1}-n_{d-k+1})}\approx 1.\] Let $c_n=\frac1{\sqrt{N_{d-k,\lambda}}}$ when $n\in\mathcal{C}$, and $c_n=0$, otherwise. Define the eigenfunction $e_\lambda(x)=\sum_{n\in\mathcal{E}}c_ne^{in\cdot x}$. Then $\sum_{n\in\mathcal{E}}|c_n|^2=1$, by definition. Then
\begin{align*}\int_\gamma |e_\lambda|^2d\sigma &=\sum_{m\in\mathcal{E}}\sum_{n\in\mathcal{E}}c_m\bar{c}_n\int_\Sigma e^{i(m-n)\cdot x}d\sigma\\
&\approx\frac1{N_{d-k,\lambda}}\sum_{m\in\mathcal{C}}\sum_{n\in\mathcal{C}}\frac{\sin{(m_1-n_1+...+a_{d-k+1}(m_{d-k+1}-n_{d-k+1}))}}{m_1-n_1+...+a_{d-k+1}(m_{d-k+1}-n_{d-k+1})}\\
&\approx N_{d-k,\lambda}\end{align*}
Thus  \eqref{Tdsharp} is proved.

\subsection{Proof of Theorem \ref{prop7}}
Let $\frac{p}{q}\in\mathbb{Q}$, $p,q\in\mathbb{Z}$, $q\ne0$. The \textbf{height} of the rational number $\frac{p}{q}$ with $gcd(p,q)=1$ is defined by
\[H(p/q)=\max\{|p|,|q|\}.\]
Clearly, $H(0)=1$ and $H(n)=|n|$ if $n\in\mathbb{Z}\setminus\{0\}$. If $gcd(p,q)\ge1$, then \[H(p/q)=\frac{\max\{|p|,|q|\}}{gcd(p,q)}\le \max\{|p|,|q|\}.\] Let $x,y\in \mathbb{Q}$. Then by definition, we have the following elementary arithmetic properties
\[H(-x)=H(x)\]
\[H(x^{-1})=H(x),\ x\ne0\]
\[H(x+y)\le 2H(x)H(y)\]
\[H(xy)\le H(x)H(y).\]
Now we prove the following basic lemmas by using the arithmetic properties of the height function.

\begin{lemma}\label{center}
  Let $d\ge2$. Let $A_i=(a_{i1},...,a_{id})\in\mathbb{Z}^d$, $i=1,2,3$. Suppose that they are non-collinear, and each component $|a_{ij}|\ls R$, $i=1,2,3$ and $1\le j\le d$.
  If $X_0=(x_{01},...,x_{0d})$ is the center of the circle determined by these three points, then $X_0\in\mathbb{Q}^d$ and $H(x_{0j})\ls R^{c}$, $1\le j\le d$, for some constant $c>0$.\end{lemma}
\begin{proof}
  Let $\textbf{a}=A_1-A_2$ and $\textbf{b}=A_1-A_3$. Note that the center of the circle is the intersection of two perpendicular bisectors of the segments $A_1A_2$ and $A_1A_3$. Then by direct calculations, the equations of these two perpendicular bisectors are
  \[X(t)=\tfrac{A_1+A_2}2+t(\textbf{b}-\tfrac{\textbf{a}\cdot \textbf{b}}{|\textbf{a}|^2}\textbf{a}),\ t\in\mathbb{R},\]
   \[X(s)=\tfrac{A_1+A_3}2+s(\textbf{a}-\tfrac{\textbf{a}\cdot \textbf{b}}{|\textbf{b}|^2}\textbf{b}),\ s\in\mathbb{R}.\]
  Thus we only need to solve $t,s$ from the $d$-dimensional linear system
 \[\tfrac{A_1+A_2}2+t(\textbf{b}-\tfrac{\textbf{a}\cdot \textbf{b}}{|\textbf{a}|^2}\textbf{a})=\tfrac{A_1+A_3}2+s(\textbf{a}-\tfrac{\textbf{a}\cdot \textbf{b}}{|\textbf{b}|^2}\textbf{b})\]
   which can be solved after finite steps of elementary arithmetic. So the solution $t,s\in \mathbb{Q}$. By the arithmetic properties of the height function, $H(t)$ and $H(s)$ are $\ls R^{c_0}$ for some constant $c_0>0$. Therefore, the center $X_0\in\mathbb{Q}^d$ and each $H(x_{0j})\ls R^c$ for some constant $c>0$.
\end{proof}
\begin{lemma}\label{plane}
  Let $d\ge2$. Let $A_i=(a_{i1},...,a_{id})\in\mathbb{Z}^d$, $i=1,2,3$. Suppose that they are non-collinear, and each component $|a_{ij}|\ls R$, $i=1,2,3$ and $1\le j\le d$. Then there are $i_1,i_2\in\{1,...,d\}$ such that the equation of the 2-plane passing through the three points has the form
  \[X=V_0+x_{i_1}V_1+x_{i_2}V_2\]
  where $X=(x_1,...,x_d)$, and $V_0,V_1,V_2\in\mathbb{Q}^d$ and each of their components has height $\ls R^c$ for some constant $c>0$.
\end{lemma}
\begin{proof}
  Let $\textbf{a}=A_1-A_2$ and $\textbf{b}=A_1-A_3$. Then the equation of the 2-plane can be
  \[X(t,s)=A_1+t\textbf{a}+s\textbf{b},\ \ \ t,s\in \mathbb{R}.\]
  Consider the $d\times2$ matrix $M=[\textbf{a}\ \textbf{b}]$. Note that rank($M$)=2. So there are two linear independent rows, say, the $i_1$-th and $i_2$-th rows. Thus we can represent $(t,s)$ by $(x_{i_1},x_{i_2})$ after finite steps of elementary arithmetic, and then the equation of the plane changes to
   \[X=V_0+x_{i_1}V_1+x_{i_2}V_2\]
   where $V_0,V_1,V_2\in\mathbb{Q}^d$. By the arithmetic properties of the height function, each of their components has height $\ls R^c$ for some constant $c>0$.
\end{proof}

\begin{lemma}\label{sphere}
  Let $d\ge2$ and $1\le k\le d-1$. If  $R S^{k}(X_0)\subset \mathbb{R}^d$ is a $k$-sphere embedded in $\mathbb{R}^d$ with center $X_0\in\mathbb{R}^d$ and radius $R>1$, then
  \[\#\mathbb{Z}^d\cap R S^k(X_0)\ls R^{k-1+\eps},\ \forall\eps>0\]
  where the constant is independent of the embedding, $X_0$ and $R$.
\end{lemma}
\begin{proof}
   By the translation invariance of $\mathbb{Z}^d$, we may assume $X_0\in[-1,1]^d$, which implies that \[|x_j|\ls R,\ 1\le j\le d\] for any point $X=(x_1,...,x_d)\in R S^k(X_0)$. To count the lattice points in $R S^k(X_0)$, we can restrict $(k-1)$ coordinates of $X$ to some integers with absolute values $\ls R$ (there are $\ls R^{k-1}$ choices in total), and reduce to counting lattice points on the circles. So it suffices to prove the case $k=1$:
  \[\#\mathbb{Z}^d\cap R S^1(X_0)\ls R^{\eps},\ \forall\eps>0\]
  where the constant is independent of the embedding, $X_0$ and $R$.
  We may assume that there are at least $3$ lattice points on the circle $R S^1(X_0)$, otherwise we are done. By Lemma \ref{center}, the center $X_0\in\mathbb{Q}^d$, and each component has height $\ls R^c$ for some $c>0$. By Lemma \ref{plane} and a change of variables, the 2-plane containing the circle may have the form: \[X=V_0+x_1V_1+x_2V_2\] where $V_0,V_1,V_2\in\mathbb{Q}^d$ and each of their components has height $\ls R^c$ for some constant $c>0$. To estimate the number of lattice points $X=(x_1,x_2,...,x_d)$ on the circle, we only need to count the number of integer solutions $(x_1,x_2)\in \mathbb{Z}^2$ to the equation
   \[|x_1V_1+x_2V_2+V_0-X_0|^2=R^2.\]
Since all the components of $X_0,\ V_0,\ V_1,\  V_2$ have heights $\ls R^c$, it is equivalent to an equation with integer coefficients:
\[|x_1\tilde V_1+x_2\tilde V_2+\tilde V_0-\tilde X_0|^2=R_1\]
where $\tilde V_1,\ \tilde V_2,\ \tilde V_0,\ \tilde X_0\in \mathbb{Z}^d$ and $|R_1|\ls R^{c_d}$ for some constant $c_d>0$.
Expanding the square
\[Ax_1^2+Bx_2^2+2Cx_1x_2+2Dx_1+2Ex_2+F=0\]
where $A,B,C,D,E\in\mathbb{Z}$, and $F=|\tilde V_0-\tilde X_0|^2-R_1$. These coefficients are all bounded by $R^{c_d}$ for some constant $c_d>0$.
Completing the square gives
\[(AB-C^2)(Ax_1+Cx_2+D)^2+((AB-C^2)x_2+(AE-CD))^2=K\]
where $K=(AE-CD)^2-(AF-D^2)(AB-C^2)$.

  Note that $AB-C^2=|\tilde V_1|^2|\tilde V_2|^2-(\tilde V_1\cdot\tilde V_2)^2>0$. If $AB-C^2=PQ^2$ for some $P,Q\in\mathbb{Z}$ and $P>0$ is square free, then by the following substitutions
  \[x=(AB-C^2)x_2+(AE-CD)\]
  \[y=Q(Ax_1+Cx_2+D)\]
  it is reduced to counting the number of integer solutions $(x,y)\in\mathbb{Z}^2$ to the equation
  \begin{equation}
    x^2+Py^2=K
  \end{equation}
  where $P>0$ is squarefree and $K\in\mathbb{Z}$ satisfies $|K|\ls R^{c_d}$ for some constant $c_d>0$. Namely, we need to estimate the number $r_P(K)$ of representations of an integer $K$ by the quadratic form $x^2+Py^2$. Recall that  (see e.g. \cite[page 32]{Ye2012})
  \[r_P(K)\le 6\tau(K)\]
  where $\tau(K)$ is the number of divisors of $K$ and satisfies $\tau(K)\ls |K|^\eps,\ \forall\eps>0$. See e.g. \cite[page 296]{Apo76} for the divisor bound. Therefore,
   \[\#\mathbb{Z}^d\cap R S^1(X_0)\ls R^{\eps},\ \forall\eps>0.\]
\end{proof}
\noindent\textbf{Proof of \eqref{Tdkrational}.} When $1\le k\le d-2$, by the proof of \eqref{Tdk} we only need to estimate the maximal number of lattice points in the intersection of $k$ bands:
 \[\max_{m\in\mathcal{E}}\#\bigcap_{j=1}^kB_{p_j}(u_j,m)\]
 in \eqref{Tdkpf0}. Since we are assuming $u_1,...,u_k$ are fixed and rational,  the lattice points in the intersection of $k$ unit bands that are $\nu_{k,d}$-transverse must lie in a finite number (depending on $u_1,...,u_k$) of $(d-k)$-dimensional affine planes. Note that the intersection of each $(d-k)$-dimensional affine plane and the sphere $\lambda S^{d-1}$ equals to an embedded lower dimensional sphere $R S^{d-k-1}(X_0)\subset \mathbb{R}^d$ with  center $X_0\in\mathbb{R}^d$ and radius $R\le \lambda$. By Lemma \ref{sphere}, the number of lattice points in $R S^{d-k-1}(X_0)$ is uniformly bounded by $\lambda^{d-k-2+\eps}$, independent of its center. Therefore, the number of lattice points
\[\max_{m\in\mathcal{E}}\#\bigcap_{j=1}^kB_{p_j}(u_j,m)\ls \lambda^{d-k-2+\eps}2^{p_1+\cdot\cdot\cdot +p_k},\ \forall\eps>0\] where the constant depends on $u_1,...,u_k$. So the proof is completed by using \eqref{Tdkpf0}.\\

\noindent\textbf{Proof of \eqref{Tdk1ratonal}. } When $k=d-1$, similarly by the proof of \eqref{Tdk} we can get \[\max_{m\in\mathcal{E}}\#\bigcap_{j=1}^{d-1}B_{p_j}(u_j,m)\ls 2^{p_1+\cdot\cdot\cdot +p_k},\] with fixed rational $u_1,...,u_k$, as any straight line intersects the sphere at no more than 2 points. Arguing along the  proof of \eqref{Tdk} only gives an upper bound of some power of $\log\lambda$. To prove the expected uniform bound, we need to use Lemma \ref{hilbert}. The following argument is the generalization of the proof of \eqref{T2rational} in 2 dimension.

Without loss of generality, we set $\Sigma=(x_1,...,x_{d-1},a_1x_1+...+a_{d-1}x_{d-1})$, $|x_j|\le1$, $|a_j|\le1$, $j=1,...,d-1$. Then
\begin{align*}
  \int_\Sigma|e_\lambda|^2d\sigma&=\sum_m\sum_nc_m\bar{c}_n\int_\Sigma e^{i(m-n)\cdot x}d\sigma\\
  &\ls \Big|\sum_m\sum_nc_m\bar{c}_n\prod_{j=1}^{d-1}\frac{\sin((m_j+a_jm_d)-(n_j+a_jn_d))}{(m_j+a_jm_d)-(n_j+a_jn_d)}\Big|.\\
\end{align*}

For fixed $t_1,...,t_{d-1}$, the straight line determined by $d-1$ equations: $x_j+a_jx_d=t_j$, $j=1,...,d-1$, intersects the sphere $\lambda S^{d-1}$ at at most two points, and they are separated by the hyperplane $a_1x_1+...+a_{d-1}x_{d-1}-x_d=0$ passing through their midpoint. Let
 \[L^+=\{(x_1,...,x_d)\in\mathcal{E}:a_1x_1+...+a_{d-1}x_{d-1}-x_d\ge 0\},\]
 \[L^-=\{(x_1,...,x_d)\in\mathcal{E}:a_1x_1+...+a_{d-1}x_{d-1}-x_d< 0\},\]
 and then $\mathcal{E}=L^+\cup L^-$.

 As in the 2-dimensional proof, by splitting the sum, it suffices to estimate
 \begin{align*}I_1&=\Big|\sum_{m\in L^+}\sum_{n\in L^+}c_m\bar{c}_n\prod_{j=1}^{d-1}\frac{\sin((m_j+a_jm_d)-(n_j+a_jn_d))}{(m_j+a_jm_d)-(n_j+a_jn_d)}\Big|\\
 &\ls \Big(\sum_{m\in L^+}\Big(\sum_{n\in L^+}\bar{c}_n\prod_{j=1}^{d-1}\frac{\sin((m_j+a_jm_d)-(n_j+a_jn_d))}{(m_j+a_jm_d)-(n_j+a_jn_d)}\Big)^2\Big)^\frac12.\end{align*}
 If $a_j=p_j/q_j\in\mathbb{Q}$, $gcd(p_j,q_j)=1$, $|p_j|\le |q_j|$, then for any fixed $t_1,...,t_{d-1}\in \mathbb{Z}$, there is at most one $n\in L^+$ such that $q_jn_j+p_jn_d=t_j$, $j=1,...,d-1$. So we may define $a_{t_1,...,t_{d-1}}:=\bar{c}_n$ if there is one such $n\in L^+$, and  define $a_{t_1,...,t_{d-1}}:=0$, otherwise. Therefore, recalling the operator $T_\mu$ in Lemma \ref{hilbert} we have
 \begin{align*}
   I_1&\ls \Big(\sum_{s_1,...,s_{d-1}\in\mathbb{Z}}\Big(\sum_{t_1,...,t_{d-1}\in\mathbb{Z}}a_{t_1,...,t_{d-1}}\prod_{j=1}^{d-1}\frac{\sin((s_j-t_j)/q_j)}{s_j-t_j}\Big)^2\Big)^\frac12\times\prod_{j=1}^{d-1}|q_j|\\
    &\ls (\sum_{t_1,...,t_{d-1}\in\mathbb{Z}}|a_{t_1,...,t_{d-1}}|^2)^\frac12\prod_{j=1}^{d-1}|q_j|\|T_{1/q_j}\|_{l^2\to l^2}\\
   &\ls \prod_{j=1}^{d-1}|q_j|.
 \end{align*}
 So we get
 \[\int_\Sigma|e_\lambda|^2d\sigma\ls \prod_{j=1}^{d-1}|q_j|.\]

\section{Explicit estimates in 3 dimension}
In this section, we prove Theorem \ref{prop4} on $\mathbb{T}^3$.  Let $e_\lambda(x)=\sum_{n\in \mathcal{E}}c_ne^{in\cdot x}$, $\sum_{n\in\mathcal{E}}|c_n|^2=1$, where $\mathcal{E}=\mathbb{Z}^3\cap \lambda S^2$. By Theorem \ref{prop3a}, it suffices to estimate $A_{1,3,\lambda}$ and $A_{2,3,\lambda}$, and prove that
\[A_{1,3,\lambda}\ls \lambda^{\frac23+\eps},\]
\[A_{2,3,\lambda}\ls \lambda^{\frac16+\eps}.\]Recall that $A_{1,3,\lambda}$ is the maximal number of lattice points in a unit band of $\lambda S^2$, and $A_{2,3,\lambda}$ is the maximal number of lattice points in the intersection of two $\nu$-transverse unit bands of $\lambda S^2$. The main idea is to decompose the bands into identical small portions with at most $\lambda^\eps$ lattice points, and then count the number of the portions. It generalizes a result of Jarnik \cite{Jar1926}, any arc on $\lambda S^1$ of length at most $c\lambda^\frac13$ contains at most 2 lattice points.

Indeed, we can use Lemma \ref{volume} below to conclude that the lattice points in a spherical region of $\lambda S^2$ must be coplanar if the convex hull of the spherical region has volume $< \frac16$. Here the \textbf{convex hull} is the smallest convex set in $\mathbb{R}^3$ that contains the region. For example, the convex hull of four  points in $\mathbb{R}^3$ which are not coplanar is the tetrahedron which has these four points as vertices.

\begin{lemma}\label{volume}
The convex hull of four lattice points that are not coplanar has volume $\geq \frac{1}{6}$.
\end{lemma}
\begin{proof}
Let $\textbf{a}, \textbf{b}, \textbf{c}, \textbf{d}\in \mathbb{Z}^3$. Then the volume of the convex hull, or the tetrahedron formed by these four points is
\[V=\frac{1}{6}|\operatorname{det}(\mathbf{a}-\mathbf{d}, \mathbf{b}-\mathbf{d}, \mathbf{c}-\mathbf{d})|.\]
Since we know these four points are not coplanar, the determinant must be a non-zero integer, which means  $V\geq \frac16$.
\end{proof}
\begin{corollary}\label{convex}
Let $E$ be a subset of $\lambda S^2$. If the convex hull of $E$ in $\mathbb{R}^3$ has volume $\ls 1$, then the number of lattice points in $E$ is $\ls \lambda^\epsilon$.
\end{corollary}
\begin{proof}
We may use 2-planes to cut the convex hull of $E$ into a finite number ($\ls 1$) of small convex sets of volume $<\frac16$, by the continuity of volume measure (or the ham sandwich theorem \cite{ST1942}). Then the lattice points in each small convex set must be coplanar by Lemma \ref{volume}. Note that the intersection of a plane and the sphere $\lambda S^2$ is a circle with radius $R\le \lambda$. Then the intersection of $\lambda S^2$ and each small convex set has $\ls \lambda^\eps$ lattice points by  Lemma \ref{sphere}. So there are $\ls \lambda^\eps$ lattice points in $E$.
\end{proof}
\begin{corollary}\label{cap14}
  If $E$ is a cap on $\lambda S^2$ with radius $\ls\lambda^\frac14$, then the number of lattice points in $E$ is $\ls \lambda^\epsilon$.
\end{corollary}
\begin{proof}
It follows from Corollary \ref{convex} and the fact that the convex hull of the $\lambda^\frac14$-cap has volume $ \approx\lambda^{-\frac12}\times(\lambda^\frac14)^2\ls 1$.
\end{proof}
We use the following convention in this section. Consider a unit band on $\lambda S^2$. The band is between two parallel planes, and for simplicity we may assume both planes have nonempty intersections with the sphere $\lambda S^2$. Indeed, other cases can be trivially reduced to this case. The boundary of the band consists of two circles, which are the intersections of  the sphere and the two parallel planes. If the largest radius of the circles is $R$, we call the \textbf{radius} of the band is $R$. The distance between the two parallel planes is called the \textbf{width} of the band.  By definition, the width of the unit band is comparable to 1. Moreover, the spherical distance between the two circles is called the \textbf{spherical\ width} of the band. A \textbf{band sector} is a portion of the unit band with respect to a central angle. The \textbf{length of the band sector} is the radius of the band times the central angle.

For example, the unit band in Figure \ref{fig31} below is the set \[\{(x,y,z)\in \mathbb{R}^3: x^2+y^2+z^2=\lambda^2, a\le z\le a+1 \}\] for some value $a$. If the band contains the equator (e.g. $-1<a<0$), we may split it into two bands with respect to the equator. For simplicity, we may only consider the case $a\ge0$, since our goal is counting lattice points on the unit band. It is formed by rotating the arc in Figure \ref{fig32} with respect to the z-axis, and as in Figure \ref{fig32}, the radius of the band is $R=\sqrt{\lambda^2-a^2}$, the spherical width of the band is equal to the length of the arc. Finally, a band sector is formed by rotating the arc with respect to z-axis for some angle $\theta\in[0,2\pi)$, and the length of the band sector is $R\theta$. See Figure \ref{fig33}.
\begin{figure*}[htbp]
  \centering
  \subfigure[A unit band]{
    \includegraphics[width=0.6\textwidth]{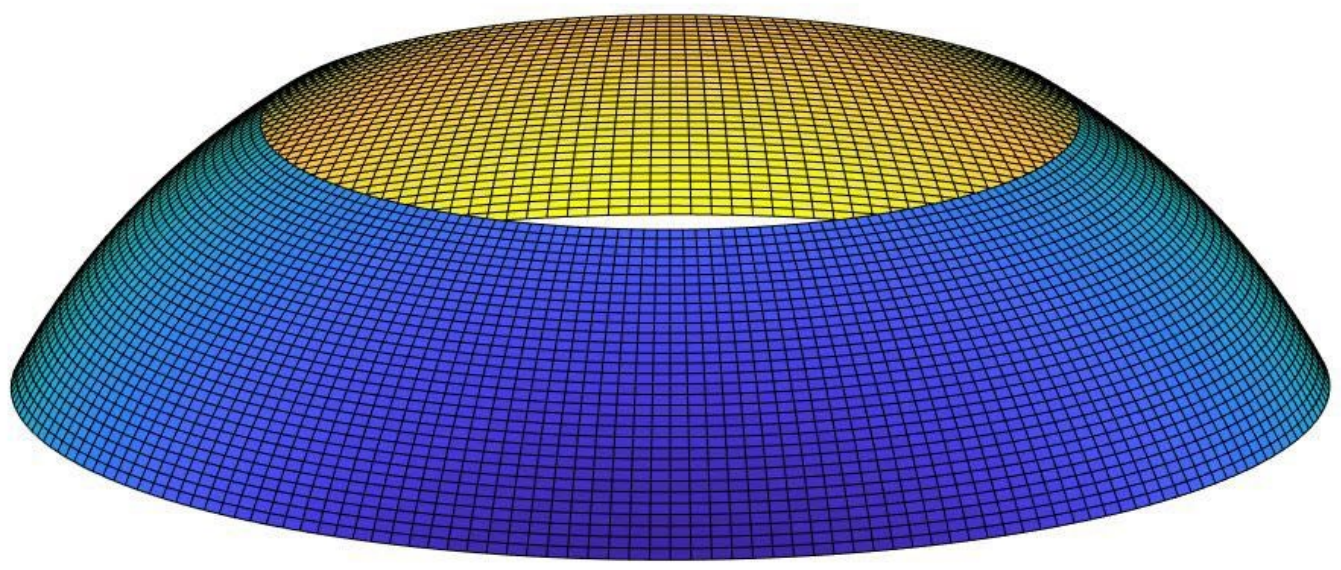}
    \label{fig31}
  }
  \hfill
  \subfigure[The corresponding arc]{
    \includegraphics[width=0.6\textwidth]{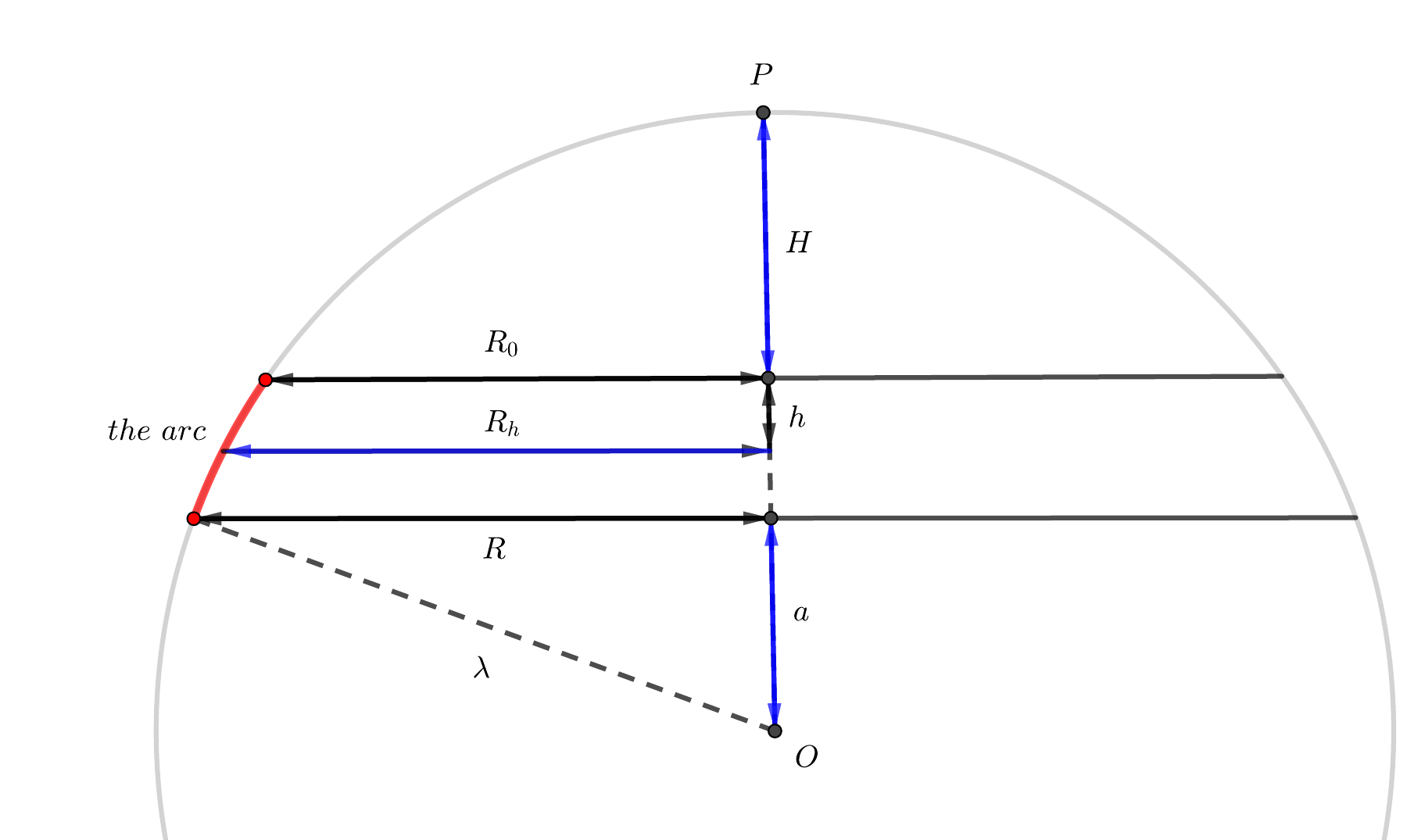}
    \label{fig32}
  }
  \caption{}
  \label{fig3132}
\end{figure*}
\begin{figure*}[htbp]
  \centering
  \subfigure[The band sector $ABB_0A_0$ with the angle $\theta$]{%
    \includegraphics[width=0.5\textwidth]{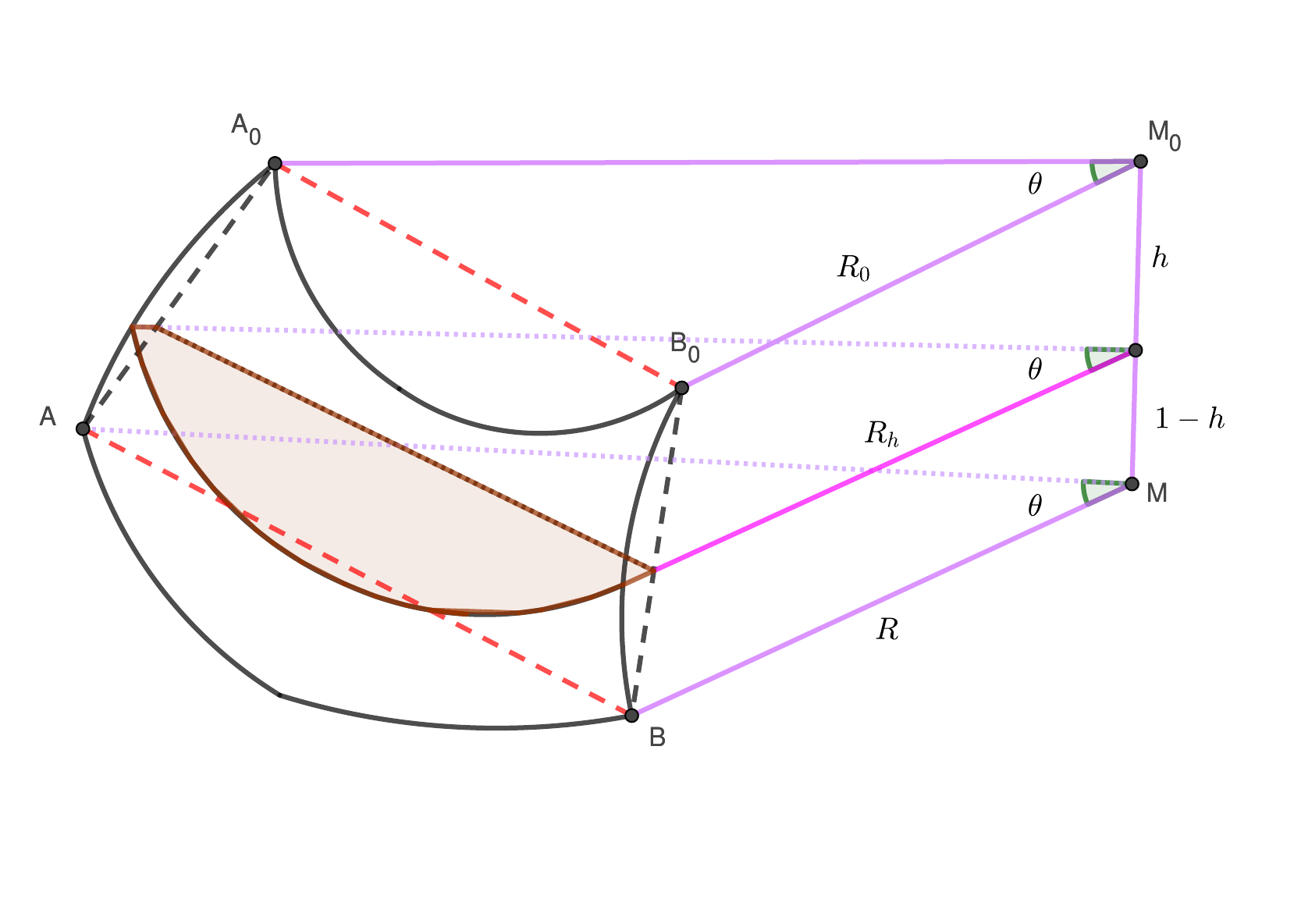}%
    \label{fig33}%
  }%
  \hfill
  \subfigure[A cross section of the convex hull $ABB_0A_0$]{%
    \includegraphics[width=0.5\textwidth]{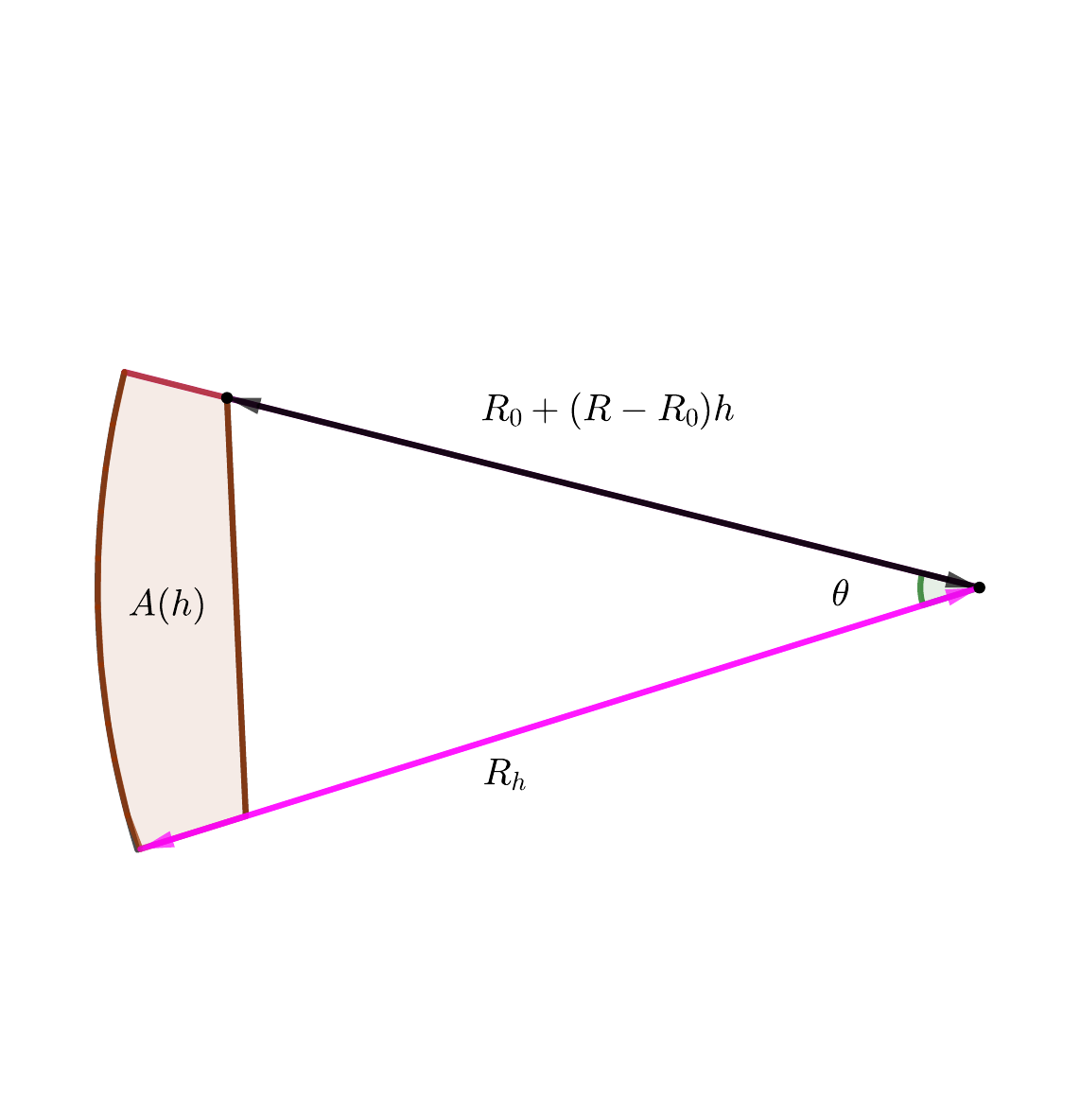}%
    \label{fig34}%
  }%
  \caption{}
  \label{fig3334}
\end{figure*}
\begin{lemma}\label{width}
If a unit band on $\lambda S^2$ has radius $R$, then $R\gs \lambda^\frac12$ and its spherical width $\approx \frac{\lambda}{R}\ls \lambda^\frac12$,  and its area $\approx\lambda$.
\end{lemma}
\begin{proof}
From the discussion above, we know the spherical width of the band is equal to the arc length in Figure \ref{fig32},  which is comparable to the length of the chord:
\begin{align*}
l=\sqrt{1+(R-R_0)^2}.
\end{align*}
Since
\[R_0^2=\lambda^2-(1+\sqrt{\lambda^2-R^2})^2\geq 0\] we have
\[|R-R_0|=\frac{1+2\sqrt{\lambda^2-R^2}}{R+R_0}\approx \frac{1+2\sqrt{\lambda^2-R^2}}{R}.\]
When $R\approx \lambda $, we have $|R-R_0|\ls 1$ and $l\approx 1\approx \frac{\lambda}{R}$. When $R\le\frac1{2}  \lambda $, we have $|R-R_0|\approx \frac{\lambda}{R}$ and $l\approx \frac{\lambda}{R}$. Thus the spherical width $\approx \lambda/R$. As in Figure \ref{fig32},
\[R=\sqrt{\lambda^2-(\lambda-H-1)^2}\ge \sqrt{\lambda^2-(\lambda-1)^2}\approx \lambda^\frac12,\]
 then the spherical width $\ls \lambda^\frac12$.

Since the band has length $\approx R$,
\[{\rm  Area\ of\ the\ band} \approx {\rm length}\times {\rm spherical\  width}\approx R\times \frac{\lambda}{R}=\lambda\]
which is independent of $R$.
\end{proof}
\begin{lemma}\label{sector}
  Let $E$ be a band sector on $\lambda S^2$ with radius $R$ and central angle $\theta$. If $\theta\approx R^{-\frac23}$ and $R\gs \lambda^\frac34$, then the volume of its convex hull $\ls 1$.
\end{lemma}
\begin{proof}
  As shown in Figure \ref{fig33}, the convex hull of the band sector $ABB_0A_0$ is formed by the plane $ABB_0A_0$, two parallel planes ($A_0B_0M_0$ and $ABM$), two intersecting planes ($AA_0M$ and $BB_0M$), and the band sector. To compute the volume of the convex hull, it is convenient to consider its cross section first. In Figure \ref{fig33}, the cross section is a part of the circular sector with radius $R_h$ and angle $\theta$, where an isosceles triangle is removed. A direct computation on the right trapezoid $AMM_0A_0$ (or $BMM_0B_0$) shows that the two legs of this isosceles triangle have equal length $R_0+(R-R_0)h$. See Figure \ref{fig34}). Now  we compute the area of the cross section (the shaded region in Figure \ref{fig34}), and then integrate over $h\in[0,1]$. From Figure \ref{fig32}, we have \[R_h^2=\lambda^2-(\lambda-H-h)^2.\]
  Here we have $H\le \lambda-1$ and
  \[H=\lambda-1-\sqrt{\lambda^2-R^2}\approx \lambda^{-1}R^2\gs \lambda^\frac12\]
  by the assumption that $R\gs \lambda^\frac34$.
  Then the area of the shaded region in Figure \ref{fig34} is
  \[A(h)=\frac12\theta(\lambda^2-(\lambda-H-h)^2)-\frac12(R_0+(R-R_0)h)^2\sin\theta.\]
  Note that
  \begin{align*}&R^2=\lambda^2-(\lambda-H-1)^2=2\lambda H-H^2+2\lambda-2H-1\\
  &R_0^2=\lambda^2-(\lambda-H)^2=2\lambda H-H^2\end{align*}
\[(2\lambda H-H^2+\lambda-H)-RR_0=\frac{\lambda^2}{RR_0+(2\lambda H-H^2+\lambda-H)}\ls \lambda H^{-1}.\]
  Recall that $\sin\theta\ge\theta-\frac16\theta^3$ when $0\le\theta<1$.
  Thus the volume of the convex hull is
  \begin{align*}
    V&=\int_0^1A(h)dh=\frac\theta2(2\lambda H-H^2+\lambda-H-\tfrac13)-\frac{\sin\theta}2\frac13(R^2+R_0^2+RR_0)\\
    &\le \frac\theta2(2\lambda H-H^2+\lambda-H)-\frac{1}2(\theta-\frac16\theta^3)(2\lambda H-H^2+\lambda-H+O(\lambda H^{-1}))\\
    &\ls \theta^3\lambda H+\theta\lambda H^{-1}\\
    &\approx \theta^3R^2+\theta\lambda^2 R^{-2}\\
    &\approx 1
  \end{align*}by the assumption that $\theta\approx R^{-\frac23}$ and $R\gs \lambda^\frac34$.
\end{proof}

\noindent\textbf{Proof of \eqref{T3k1n}.} We need to count the maximal number of lattice points on a unit band on $\lambda S^2$. Let $R$ be the radius of the unit band. We consider 3 cases.

\noindent(1) If $R\ls \lambda^\frac12$, then the band can be covered by one $\lambda^\frac12$-cap, which has $\ls \lambda^\frac12$ lattice points by \eqref{N2}.

\noindent(2) If $\lambda^\frac12\ls R\ls \lambda^\frac34$, then the length of the band is $\approx R\gs \lambda^\frac12$ and the spherical width $\gs \lambda^\frac14$ by Lemma \ref{width}. So we can efficiently cover the band by caps with radius $\approx \lambda^\frac14$. Recall that the area of the band is $\approx \lambda$ by Lemma \ref{width}. Thus, the number of covering  caps is $\approx \lambda^\frac12$ and each cap has $\ls \lambda^\eps$ lattice points by Corollary \ref{cap14}, so the band contains $\ls \lambda^{\frac12+\eps}$ lattice points.

\noindent(3) If $R\gs \lambda^\frac34$, then we may decompose the band into identical band sectors with the central angle $\theta\approx R^{-\frac23}$. The number of the band sectors is $\approx R^\frac23$ and each band sector has $\ls \lambda^\eps$ lattice points by Lemma \ref{sector} and Corollary \ref{convex}, so the band contains $\ls \lambda^{\frac23+\eps}$ lattice points.

So we finish the proof of \eqref{T3k1n}.
\begin{remark}{\rm
 From the proof, we see that the worst case is the case (3), where the unit band is close to the equator. If one can get an improvement in this case, the bound in \eqref{T3k1n} can be improved.}
\end{remark}

As in Figure \ref{fig32}, a normal line of the parallel planes of the band may pass through the origin $O$ and intersect the sphere $\lambda S^2$ at two points. We denote the point nearest to the band by $P$ and it is called the \textbf{north pole} of the band (see Figure \ref{fig32}). The normal line $OP$ is called the \textbf{axis of the band}.
\begin{lemma}\label{pole}
  If the radius of a unit band is $R$, then the spherical distance between its north pole $P$ and any point on the band is $\ls R$.
\end{lemma}
\begin{proof}
  From Figure \ref{fig32}, the spherical distance between the north pole $P$ and the point on the band is $\ls \sqrt{(H+1)^2+R^2}$. Then the lemma follows from the fact that
   \[H\le\lambda-\sqrt{\lambda^2-R^2}\ls \lambda^{-1}R^2\ls R.\]
\end{proof}
\noindent\textbf{Proof of \eqref{T3k2n}.}
  Fix $\nu=\nu_{2,3}=\frac12$. Let $\mathcal{B}$ and $\mathcal{B}'$ be two $\nu$-transverse unit bands on $\lambda S^2$. Let $P$, $P'$ be the north poles of the two bands. Let $\alpha\in(0,\frac{\pi}2]$ be the angle between the axes of the two bands. Assume that $\mathcal{B}\cap\mathcal{B}'\ne\emptyset$. By transversality, the angle $\alpha\gs 1$ and then
  \[d(P,P')\approx \lambda\]
  where $d(\cdot,\cdot)$ is the spherical distance on $\lambda S^2$. This implies that at least one of the bands (called $\mathcal{B}$) must have radius $\approx \lambda$ and spherical width $\approx 1$. Indeed, if $x_0\in \mathcal{B}\cap\mathcal{B}'$, then by the triangle inequality
  \[d(x_0,P)+d(x_0,P')\ge d(P,P')\approx \lambda.\]
  So at least one term on the left hand side is comparable to $\lambda$. By Lemma \ref{pole}, at least one of the bands must have radius $\approx \lambda$. Then by Lemma \ref{width}, its spherical width $\approx 1$.

  The intersection $\mathcal{B}\cap\mathcal{B}'$ is contained in one band sector of $\mathcal{B}$ with length $\ls \lambda^\frac12$, since the spherical width of  $\mathcal{B}'$ is $\ls \lambda^\frac12$ by Lemma \ref{width}. As in the third case of the proof of \eqref{T3k1n}, we decompose the band $\mathcal{B}$ into $\approx \lambda^\frac23$ small identical band sectors with the central angle $\theta\approx \lambda^{-2/3}$.
  The length of each band sector is $\approx \lambda^\frac13$, so we can cover $\mathcal{B}\cap\mathcal{B}'$ by $\ls\lambda^\frac16$ these band sectors. By Lemma \ref{sector}, each band sector has $\ls \lambda^\eps$ lattice points. So the intersection $\mathcal{B}\cap\mathcal{B}'$ contains $\ls \lambda^{\frac16+\eps}$ lattice points. This completes the proof of \eqref{T3k2n}.

\begin{remark}{\rm Note that the arclength $\lambda^\frac13$ in Jarnik's result \cite{Jar1926} is not optimal. Indeed, Cilleruelo and C\'ordoba \cite{CC1992} proved that for any $\delta<\frac12$, arcs  of length $\lambda^\delta$ contain at most $M(\delta)$ lattice points and in \cite{CG2007} it is conjectured that this remains true for any $\delta<1$. If we consider the generalization in $\lambda S^2$, we may naturally expect that the length of each small band sector in the proof of \eqref{T3k1n} can be chosen to be larger than $R^\frac13$ and it still contains $\ls \lambda^\eps$ lattice points. Clearly this will lead to better estimates on $A_{1,3,\lambda}$ and $A_{2,3,\lambda}$.
  }
\end{remark}

\bibliographystyle{plain}

\end{document}